\documentclass[11pt]{article}

\usepackage[english]{babel}
\usepackage[latin1]{inputenc} 

\usepackage[centertags]{amsmath}
\usepackage{amsfonts}
\usepackage{amssymb}
\usepackage{amsthm}
\usepackage{cite}
\usepackage{geometry}
\usepackage{indentfirst}

\usepackage{bussproofs}

\usepackage{epsfig}
\usepackage{hhline}
\usepackage{tikz}

\date{}

\newtheorem{defi}{Definition}
\newtheorem{prop}{Proposition}
\newtheorem{theo}{Theorem}
\newtheorem{rem}{Remark}
\newtheorem{cor}{Corollary}
\newtheorem{lem}{Lemma}

\begin{document}


\title{
      \textbf{Rough sets semantics for the three-valued extension of first-order Priest's da Costa logic}
}

\author{
	{\large  \textsc{Jos\'e Luis Castiglioni}}\thanks{jlc@mate.unlp.edu.ar}\\
	{\small Universidad Nacional de La Plata, La Plata, Argentina} \\
	{\large \textsc{Rodolfo C. Ertola-Biraben}}\thanks{rcertola@unicamp.br} \\
	{\small Universidade Estadual de Campinas, Campinas, Brazil} \\
}

\maketitle


\vspace{1cm}


\begin{abstract}
We provide a rough sets semantics for the three-valued extension of first-order Priest's da Costa logic, which we studied in [Castiglioni, J.L. and Ertola-Biraben, R.C. Modalities combining two negations. {\em Journal of Logic and Computation} 11:341--356, 2024]. 
This semantics follows the usual pattern of the semantics for first-order classical logic.
\end{abstract}

\section{Introduction}

In this paper we will consider the first-order logic {\bf{ID$_3$}} whose language consists in a denumerable set of individual variables and a non-empty set of $n$-ary relation letters, connectives $\{\land, \lor, \neg, D, \bot \}$ with arity $(2, 2, 1, 1, 0)$ and quantifiers $\forall$ and $\exists$. 
The connective $D$ stands for the dual of intuitionistic negation, using the notion of duality in the sense already present in \cite{ES}. 
Formulas are defined as usual.
In \cite{MCTN} we studied the same three-valued first-order logic extended with propositional letters in its language, where it is called R+S+cS.

In this paper we present a rough sets semantics for {\bf{ID$_3$}}.

We start Section 2 presenting {\bf{ID$_3$}}. 
Afterwards, we give a logic, which is equivalent by translation to {\bf{ID$_3$}}, in the propositional language $\{\land, \lor, \neg, \Box, \bot \}$. 
As expected, the formula $\Box \alpha$ may be read as ``$\alpha$ is necessary".

In Section 3 we recall the algebraic semantics and Kripke models for {\bf{ID$_3$}}, which will be used in the next sections.

For the reader not acquainted with rough sets, Section 4 begins stating the basic information in order to render this paper self-contained. 
Afterwards, we introduce the announced rough set semantics.

Relating the rough set semantics with the Kripke models, in Section 5 we prove soundness and completeness of {\bf{ID$_3$}} using the results of soundness and completeness proved in \cite{MCTN}.

\section{The logic {\bf{ID$_3$}}}

The following are the usual Gentzen Natural Deduction rules for conjunction and disjunction (see \cite[p. 186]{GG}):

\begin{center}
 \AxiomC{$\alpha$}
 \AxiomC{$\beta$}
 \LeftLabel{{\bf($\land$I)}\ }
 \RightLabel{,}
 \BinaryInfC{$\alpha \land \beta$}
 \DisplayProof \ \ 
 \AxiomC{$\alpha \land \beta$}
 \LeftLabel{{\bf($\land$E$_l$)}\ }
 \RightLabel{,}
 \UnaryInfC{$\alpha$}
 \DisplayProof \ \ 
\AxiomC{$\alpha \land \beta$}
 \LeftLabel{{\bf($\land$E$_r$)}\ }
 \RightLabel{,}
 \UnaryInfC{$\beta$}
 \DisplayProof
\end{center}

\begin{center}
  \AxiomC{$\alpha$}
 \LeftLabel{{\bf($\lor$I$_l$)}\ }
 \RightLabel{,}
 \UnaryInfC{$\alpha \lor \beta$}
 \DisplayProof \ \ 
\AxiomC{$\beta$}
 \LeftLabel{{\bf($\lor$I$_r$)}\ }
 \RightLabel{,}
 \UnaryInfC{$\alpha \lor \beta$}
 \DisplayProof
 \AxiomC {$\alpha \lor \beta$}
 \AxiomC {[$\alpha$]}
 \noLine
 \UnaryInfC {$\gamma$}
 \AxiomC {[$\beta$]}
 \noLine
 \UnaryInfC{$\gamma$}
 \LeftLabel{{\bf($\lor$E)}\ }
 \RightLabel{.} 
 \TrinaryInfC{$\gamma$}
\DisplayProof \ \ 
 \end{center}

The usual Gentzen Natural Deduction rules for intuitionistic negation are as follows:

\begin{center}
\AxiomC {[$\alpha$]}
 \noLine
 \UnaryInfC {$\bot$}
  \LeftLabel{{\bf($\neg$I)}}
 \RightLabel{,}
 \UnaryInfC{$\neg \alpha$}
 \DisplayProof \ \ 
 \AxiomC{$\alpha$}
 \AxiomC{$\neg \alpha$}
 \LeftLabel{{\bf($\neg$E)}}
 \RightLabel{,}
 \BinaryInfC{$\bot$}
 \DisplayProof
 \AxiomC{$\bot$}
 \LeftLabel{(EASQ)}
 \RightLabel{.}
 \UnaryInfC {$\alpha$}
 \DisplayProof
\end{center}

We will use the following rules for the dual of intuitionistic negation (see \cite[p. 172]{GP}):

\begin{center}
 \AxiomC{}
 \LeftLabel{{\bf($D$I)\ }}
 \RightLabel{}
 \UnaryInfC{$\alpha \lor D\alpha$}
 \DisplayProof \ \ and \ \ 
 \AxiomC{$D\alpha$}
 \AxiomC{}
 \UnaryInfC{$\alpha \lor \beta$}
 \LeftLabel{{\bf($D$E) \ }} \ 
 \RightLabel{.}
 \BinaryInfC{$\beta$}
 \DisplayProof
\end{center}

\begin{rem}
The given logic with also the usual rules for the conditional appears in \cite{GP}, \cite{SP}, and \cite{MCTN}. 
There is a previous version in \cite{GM} where only derivable formulas are considered (see also \cite{DOW}).
In \cite[p. 26]{RZ} there appears the suggestion to read $\neg$ and $D$ as ``it is false that" and ``not", respectively.
\end{rem}

\begin{rem}
There is a similar system in \cite{AMK}, where the authors use the symbol $+$ for the dual of intuitionistic negation and rules (+I) and (+E), that is, 

 \AxiomC{$D \vdash T$}
 \AxiomC{$A \vdash C$}
 \LeftLabel{}
 \RightLabel{}
 \BinaryInfC{$D \vdash  + A$}
 \DisplayProof  \ and  \ \ 
 \AxiomC{$\Gamma \vdash + A$}
 \AxiomC{$\Gamma, T \vdash A$}
 \LeftLabel{} \ 
 \RightLabel{,}
 \BinaryInfC{$\Gamma \vdash B$}
 \DisplayProof

respectively (note that in the rule (+I) the letter $D$ is used as a condition). 
Moreover, the authors remark that `` the condition in (+I), namely $D$ in the premise $D \vdash T$ and in the consequent $D \vdash + A$ and $A$ in the premise $A \vdash C$ must be a single formula, not a set of formulas, is crucial to our formalization."
\end{rem}

In order to obtain the extension we are interested in, we add the following rules:

\begin{center}
\AxiomC {}
 \LeftLabel{{\bf(S)}}
 \RightLabel{,}
 \UnaryInfC{$\neg \alpha \lor \neg \neg \alpha$}
 \DisplayProof \
 \AxiomC{$D \alpha$}
 \AxiomC{$D D \alpha$}
 \LeftLabel{{\bf(cS)}}
 \RightLabel{\ and} 
 \BinaryInfC{$\bot$}
 \DisplayProof
 \AxiomC{$\alpha \ \ D\alpha$}
 \LeftLabel{{\bf{(Reg)}}}
 \RightLabel{.}
 \UnaryInfC {$\beta \lor \neg \beta$}
 \DisplayProof
\end{center}

\noindent Equivalently, instead of the rules (S) and (cS), it is possible to use the rules

\begin{center}
 \AxiomC{$D \neg \alpha$}
 \LeftLabel{{\bf(S')}}
 \RightLabel{\ and} 
 \UnaryInfC{$\neg \neg \alpha$}
 \DisplayProof \
 \AxiomC{$DD \alpha$}
 \LeftLabel{{\bf(cS')}}
 \RightLabel{, respectively.} 
 \UnaryInfC{$\neg D \alpha$}
 \DisplayProof
\end{center}

\noindent Note that either (S) or (cS) imply both that $\neg D \alpha \vdash \neg D \neg D \alpha$ and $D\neg D\neg \alpha \vdash D\neg \alpha$.

Also, due to (cS), instead of (DE), it is possible to use either the rule 

\begin{center} 
\AxiomC{$D\alpha$}
  \AxiomC{}
 \UnaryInfC{$\alpha$}
 \LeftLabel{{\bf{($D$E$'$)}}}
 \RightLabel{}
 \BinaryInfC{$\bot$}
 \DisplayProof \ 
 or the rule
 \AxiomC{}
 \UnaryInfC{$\alpha$}
 
 \RightLabel{.}
 \UnaryInfC{$\neg D \alpha$}
 \DisplayProof \
\end{center}

We will use the symbol $\vdash$ in the context $\Gamma \vdash \alpha$ (where $\Gamma$ is a set of formulas and $\alpha$ is a formula) with the usual meaning, that is, indicating the existence of at least one derivation of $\alpha$ from $\Gamma$. 
We will not add a subscript to the symbol $\vdash$ as the context will make clear what logic is being meant.

\begin{lem} \label{DND}
Let $\alpha$ be any formula. 
Then, $\neg \neg D \alpha \dashv \vdash D\alpha$.
\end{lem}

\begin{proof}
The proof for $D \alpha \vdash \neg \neg D\alpha$ is straightforward. 
For the other direction, consider the following derivation:

\begin{center} 
 \AxiomC{}
 \RightLabel{($D$I)}
 \UnaryInfC{$D\alpha \lor DD\alpha$}
 \AxiomC{1}
 \noLine
 \UnaryInfC{$D\alpha$}
 \AxiomC{1}
 \noLine
 \UnaryInfC{$DD\alpha$}
 \AxiomC{2}
 \noLine
 \UnaryInfC{$DD\alpha$}
  \RightLabel{(cS)}
 \UnaryInfC{$\neg D\alpha$}
 \AxiomC{$\neg \neg D\alpha$}
 \RightLabel{($\neg$E)}
 \BinaryInfC{$\bot$}
 \RightLabel{($\neg$I)$_2$}
 \UnaryInfC{$\neg DD\alpha$}
 \RightLabel{($\neg$E)}
 \BinaryInfC{$\bot$}
 \RightLabel{(EASQ)}
 \UnaryInfC{$D\alpha$}
 \RightLabel{($\lor$E)$_1$.}
 \TrinaryInfC{$D\alpha$}
 \DisplayProof
\end{center}
\end{proof}

In general, it holds that if $M$ is an even string of modalities, then $M\neg \alpha \dashv \vdash \neg \alpha$ and  $MD \alpha \dashv \vdash D\alpha$ and if $M$ is an odd string of modalities, then $M\neg \alpha \dashv \vdash \neg \neg \alpha$ and  $MD \alpha \dashv \vdash DD\alpha$. 
As a consequence, the modalities are as in the following figure.

\begin{center}
\begin{tikzpicture}

    \tikzstyle{every node}=[draw, circle, fill=black, minimum size=3pt, inner sep=0pt]

\draw (-2,2) node (t) [label=right:$\ D\neg \ {=} \ \neg \neg$] {};
\draw (-2,1) node (m) [label=right:$\ \circ$] {};
\draw (-2,0) node (b) [label=right:$\ \neg D \ {=} \  DD$] {};
\draw (t)--(m)--(b);

\draw (2,1.5) node (t2) [label=right:$\ D$] {};
\draw (2,.5) node (b2) [label=right:$\ \neg$] {};
\draw (b2)--(t2);

\end{tikzpicture}

\smallskip

{\large Figure: Positive and negative modalities with $D$}
\end{center}

In \cite{MCTN} it is proved that the intuitionistic conditional $\alpha \to \beta$ may be defined as $\neg(\alpha \land \neg \beta) \land (D\alpha \lor \beta)$ and so also the biconditional $\alpha \leftrightarrow \beta$ is available as $(\alpha \to \beta) \land (\beta \to \alpha)$.

\begin{prop} \label{SD}
Let $\alpha$ and $\beta$ be any formulas. Then, 

(i) If $\alpha \vdash \beta$, then $D\beta \vdash D\alpha$,

(ii) If $\alpha \dashv \vdash \beta$, then $D\alpha \dashv \vdash D\beta$,

(iii) If $\alpha \dashv \vdash \beta$, then $\delta^{\beta/ \alpha} \dashv \vdash \delta$, for any formula $\delta$,

\noindent where the notation $\delta^{\beta/\alpha}$ stands for the formula that results from substituting in $\delta$ some or all occurrences of $\alpha$ for ocurrences of $\beta$.
\end{prop}

\begin{proof}
In (i) the hypothesis implies $\vdash \beta \lor D\alpha$ by ($D$I) whence supposing $D\beta$ by ($D$E) it follows that $D\alpha$. 

 Part (ii) follows by part (i).
 
 Part (iii) follows by part (ii) and similar properties in the case of the intuitionistic connectives.
\end{proof}

Note that by algebraic soundness it may be easily seen in the three-element chain that neither $\alpha \to \beta \vdash D\beta \to D\alpha$ nor $\alpha \leftrightarrow \beta \vdash D\alpha \leftrightarrow D\beta$ are the case.

Finally, the usual Gentzen quantifier rules are also included. 
As stated in the Introduction, our logic will be called {\bf{ID$_3$}}.

\subsection{A modal version of {\bf{ID$_3$}}}

Some readers may be interested in a version of the same logic using the necessity operator where the usual Necessitation rule is present.
Let us consider the logic {\bf{I$\Box_3$}} in the propositional language $\{\land, \lor, \neg, \Box, \bot \}$ with the following rules instead of the rules (DI), (DE), (S), (cS), and (Reg):

\begin{center}
 \AxiomC{}
 \LeftLabel{{\bf($\neg\Box$I)}}
 \RightLabel{,}
 \UnaryInfC{$\alpha \lor \neg \Box \alpha$}
 \DisplayProof
 \AxiomC{$\neg \Box \alpha$}
 \AxiomC{}
 \UnaryInfC{$\alpha \lor \beta$}
 \LeftLabel{{\bf($\neg \Box$E)}} 
 \RightLabel{,}
 \BinaryInfC{$\beta$}
 \DisplayProof \ 
 \AxiomC{$\neg \alpha$}
 \LeftLabel{($\Box$S)}
 \RightLabel{,}
 \UnaryInfC{$\Box \neg \alpha$}
 \DisplayProof
\AxiomC{$\neg \Box \neg \Box \alpha$}
 \LeftLabel{($\Box$cS)} 
 \RightLabel{,}  
 \UnaryInfC{$\Box \alpha$}
 \DisplayProof 
  \AxiomC{$\alpha \ \ \neg \Box \alpha$}
 \LeftLabel{($\Box$Reg)}
 \RightLabel{.}
 \UnaryInfC {$\beta \lor \neg \beta$}
 \DisplayProof
\end{center}

The intuitionistic conditional $\alpha \to \beta$ may also be defined in {\bf{I$\Box_3$}} as $\neg(\alpha \land \neg \beta) \land (\neg \Box \alpha \lor \beta)$ and so also the biconditional $\alpha \leftrightarrow \beta$ is available as $(\alpha \to \beta) \land (\beta \to \alpha)$.

\begin{lem} \label{BLB}
Let $\alpha$ be any formula. Then,

(S) $\vdash \neg \alpha \lor \neg \neg \alpha$,

(T) $\Box \alpha \vdash \alpha$, 

($\Box$DN) $\neg \neg \Box \alpha \dashv \vdash \Box \alpha$ (Double Negation for $\Box$), 

($\Box$TND) $\Box \alpha \lor \neg \Box \alpha$ (\emph{tertium non datur} for $\Box$),

(N) If $\vdash \alpha$, then $\vdash \Box \alpha$ (\emph{Necessitation}),

(4) $\Box \Box \alpha \dashv \vdash \Box \alpha$.
\end{lem}

\begin{proof}
(S) follows by ($\neg \Box$I) and ($S\Box$).
 
(T) follows by ($\neg \Box$I).

One direction of ($\Box$DN) follows by intuitionistic logic. 
For the other direction, check the following derivation.
\begin{center}
 \AxiomC{1}
 \noLine
 \UnaryInfC{$\Box \neg \Box \alpha$}
 \RightLabel{(T)}
 \UnaryInfC{$\neg \Box \alpha$}
 \AxiomC{$\neg \neg \Box \alpha$}
 \RightLabel{($\lor$E$_1$)}
 \BinaryInfC{$\bot$}
 \RightLabel{($\neg$I)$_1$}
 \UnaryInfC{$\neg \Box \neg \Box \alpha$}
 \RightLabel{($\Box$cS).}
 \UnaryInfC{$\Box \alpha$}
 \DisplayProof
\end{center}

For ($\Box$TND) check the following derivation: 
\begin{center}
 \AxiomC{}
 \RightLabel{(S)}
 \UnaryInfC{$\neg \Box \alpha \lor \neg \neg \Box \alpha$}
\RightLabel{$\lor$ commutativity}
 \UnaryInfC{$\neg \neg \Box \alpha \lor \neg \Box \alpha$} 
 \RightLabel{$\Box$DN.}
 \UnaryInfC{$\Box \alpha \lor \neg \Box \alpha$}
 \DisplayProof
\end{center}

For (N) check the following derivation:

\begin{center}
 \AxiomC{1}
 \noLine
 \UnaryInfC{$\neg \Box \alpha$}
 \AxiomC{}
 \UnaryInfC{$\alpha$}
  \RightLabel{($\lor$I)}
 \UnaryInfC{$\alpha \lor \bot$}
 \RightLabel{($\neg\Box$E)}
 \BinaryInfC{$\bot$}
 \RightLabel{($\neg$I)$_1$}
 \UnaryInfC{$\neg \neg \Box \alpha$}
 \RightLabel{($\Box$DN).}
 \UnaryInfC{$\Box \alpha$}
 \DisplayProof
\end{center}

One direction of (4) follows from (T). 
For the other direction, check the following derivation:

\begin{center}
 \AxiomC{$\Box \alpha$}
 \AxiomC{1}
 \noLine
 \UnaryInfC{$\neg \Box \Box \alpha$}
 \AxiomC{}
 \RightLabel{($\Box$TND)}
 \UnaryInfC{$\Box \alpha \lor \neg \Box \alpha$} 
  \BinaryInfC{$\neg \Box \alpha$}
 \RightLabel{($\neg$I$_1$)}
 \BinaryInfC{$\neg \neg \Box \Box \alpha$}
 \RightLabel{($\Box$DN).}
 \UnaryInfC{$\Box \Box \alpha$}
 \DisplayProof
\end{center}
\end{proof}

\begin{prop} \label{SB}
Let $\alpha$ and $\beta$ be any formulas. 
Then,

(i) If $\alpha \vdash \beta$, then $\neg \Box \beta \vdash \neg \Box \alpha$,

(ii) If $\alpha \vdash \beta$, then $\Box \alpha \vdash \Box \beta$,

(iii) If $\alpha \dashv \vdash \beta$, then $\Box \alpha \dashv \vdash \Box \beta$,

(iv) If $\alpha \dashv \vdash \beta$, then $\delta^\beta_\alpha \dashv \vdash \delta$, for any formula $\delta$.
\end{prop}

\begin{proof}

(i) \begin{center}
 \AxiomC{$\neg \Box \beta$}
 \AxiomC{}
 \RightLabel{($\neg\Box$I)}
 \UnaryInfC{$\alpha \lor \neg \Box \alpha$}
 \AxiomC{1}
 \noLine
 \UnaryInfC{$\alpha$}
 \RightLabel{(Hyp)}
 \UnaryInfC{$\beta$}
 \RightLabel{($\lor$I)}
 \UnaryInfC{$\beta \lor \neg \Box \alpha$}
 \AxiomC{1}
 \noLine
 \UnaryInfC{$\neg \Box \alpha$}
 \RightLabel{($\lor$I)}
 \UnaryInfC{$\beta \lor \neg \Box \alpha$}
 \RightLabel{($\lor$E$_1$)}
 \TrinaryInfC{$\beta \lor \neg \Box \alpha$}
 \RightLabel{($\neg \Box$E).}
 \BinaryInfC{$\neg \Box \alpha$}
 \DisplayProof
\end{center}

 Part (ii) follows from part (i) as $\Box \alpha \vdash \neg \neg \Box \alpha \vdash \neg \neg \Box \beta \vdash \Box \beta$.
 
Part (iii) follows from part (ii). 
 
 Part (iv) follows from part (iii) and similar properties in the case of the intuitionistic connectives.
\end{proof}

The modalities are as in the following figure.
Note that possibility, usually defined as $\neg \Box \neg$, in {\bf{I$\Box_3$}} is equivalent to double negation.

\begin{center}
\begin{tikzpicture}

    \tikzstyle{every node}=[draw, circle, fill=black, minimum size=3pt, inner sep=0pt]

\draw (-2,2) node (t) [label=right:$\ \neg \Box \neg \ {=} \ \neg \neg$] {};
\draw (-2,1) node (m) [label=right:$\ \circ$] {};
\draw (-2,0) node (b) [label=right:$\ \Box$] {};
\draw (t)--(m)--(b);

\draw (2,1.5) node (t2) [label=right:$\ \neg \Box $] {};
\draw (2,.5) node (b2) [label=right:$\ \neg$] {};
\draw (b2)--(t2);

\end{tikzpicture}

\smallskip

{\large Figure: Positive and negative modalities with $\Box$}
\end{center}

\subsection{Equivalence}

It is easily seen that the logics {\bf{ID$_3$}} and {\bf{I$\Box_3$}} are equivalent using the translations $D:= \neg \Box$ and $\Box := \neg D$ together with the fact that $\neg \neg D \alpha \dashv \vdash D \alpha$ and the items stated in Lemma \ref{BLB}.

Since in this subsection we will deal with two different logics having two different languages, we will use $\mathfrak{F}_D$ and $\vdash_D$ for the set of formulas and the consequence relation of the logic {\bf{ID$_3$}} and  $\mathfrak{F}_\Box$ and $\vdash_\Box$ for the logic {\bf{I$\Box_3$}}. 

We recursively define the function $()^t: \mathfrak{F}_D \to \mathfrak{F}_{\Box}$ by the uniform replacement of any ocurrence of $D$ by $\neg \Box$.
Similarly, we define the function $()^s: \mathfrak{F}_\Box \to \mathfrak{F}_D$ by the uniform replacement of any ocurrence of $\Box$ by $\neg D$. 
It is routine to check the following facts.

\begin{lem} \label{equi}
Let $\alpha \in \mathfrak{F}_D$ and $\beta \in \mathfrak{F}_\Box$. 
Then,

(i) If $\alpha \dashv \vdash_D (\alpha^t)^s$,

(ii) If $\beta \dashv \vdash_\Box (\beta^s)^t$.
\end{lem}

\begin{proof}
Part (i) follows from part (iii) of Proposition \ref{SD} and Lemma \ref{DND}. 
Similarly, part (ii) follows from part (iv) of Proposition \ref{SB} and ($\Box$DN) in Lemma \ref{BLB}.
\end{proof}

\begin{lem} \label{HTRANS}
Functions $t$ and $s$ defined above satisfy the following facts:

(i) If $\Gamma \vdash_D \alpha$, then $\Gamma^t \vdash_\Box \alpha^t$, for $\Gamma \cup \{\alpha\} \subseteq \mathfrak{F}_D$,

(ii) If $\Gamma \vdash_\Box \alpha$, then $\Gamma^s \vdash_D \alpha^s$, for $\Gamma \cup \{\alpha\} \subseteq \mathfrak{F}_\Box$.
\end{lem}

\begin{proof}
The proof is routine. 
We explicitly work the cases of the (S) and (cS) rules in the case of the $t$-translation. 
The $t$-function of a step
\AxiomC{$D\neg \alpha$}
 \RightLabel{(S)}
 \UnaryInfC{$\neg \neg \alpha$}
 \DisplayProof
is

\begin{center}
 \AxiomC{1}
 \noLine
 \UnaryInfC{$\neg \alpha^t$}
 \RightLabel{($\Box$S)}
 \UnaryInfC{$\Box \neg \alpha^t$}
 \AxiomC{$\neg \Box \neg \alpha^t$}
 \RightLabel{($\neg$E)}
 \BinaryInfC{$\bot$}
 \RightLabel{($\neg$I)$_1$.}
 \UnaryInfC{$\neg \neg \alpha^t$}
 \DisplayProof
\end{center}

The $t$-function of a step
\AxiomC{$DD \alpha$}
 \RightLabel{(cS)}
 \UnaryInfC{$\neg D \alpha$}
 \DisplayProof
is

\begin{center}
 \AxiomC{1}
 \noLine
 \UnaryInfC{$\neg \Box \neg \Box \alpha^t$}
 \RightLabel{($\Box$cS)}
 \UnaryInfC{$\Box \alpha^t$}
 \AxiomC{$\neg \Box \alpha^t$}
 \RightLabel{($\neg$E)}
 \BinaryInfC{$\bot$}
 \RightLabel{($\neg$I)$_1$.}
 \UnaryInfC{$\neg \neg \Box \alpha^t$}
 \DisplayProof
\end{center}
\end{proof}

\begin{theo}
Functions $t$ and $s$ are translations, that is,

(i) $\Gamma \vdash_D \alpha$ iff $\Gamma^t \vdash_\Box \alpha^t$, for $\Gamma \cup \{\alpha\} \subseteq \mathfrak{F}_D$,

(ii) $\Gamma \vdash_\Box \alpha$ iff $\Gamma^s \vdash_D \alpha^s$, for $\Gamma \cup \{\alpha\} \subseteq \mathfrak{F}_\Box$.

Furthermore, these translations prove that the logics {\bf{ID$_3$}} and {\bf{I$\Box_3$}} are equivalent.
\end{theo}

\begin{proof}
Suppose $\Gamma^t \vdash_\Box \alpha^t$. Then, by part (ii) in Lemma \ref{HTRANS} it follows that $(\Gamma^t)^s \vdash_D (\alpha^t)^s$ whence $\Gamma \vdash_D \alpha$ by part (i) of Lemma \ref{equi}.
\end{proof}

In the rest of the paper we will only be considering the logic {\bf{I$D_3$}}, leaving to the reader the analogous results for the logic {\bf{I$\Box_3$}}.

\section{Semantic notions}

In this section we state the algebraic and Kripke notions required in order to understand the contents of this  paper.

\subsection{Algebraic semantics for propositional ID$_3$}  \label{31}

A \emph{double p-algebra} is an algebra $(A; \wedge, \vee, \neg, D, 0, 1)$ of type $(2, 2, 1, 1, 0, 0)$ such that $(A; \wedge, \vee, 0, 1)$ is a bounded distributive lattice and $\neg$ and $D$ are the meet and join complement, respectively, that is, they satisfy $x \wedge y = 0$ iff $y \leq \neg x$ and $x \vee y = 1$ iff $Dx \leq y$, respectively, where $\leq$ is the lattice order (see \cite{JV} and \cite{TK} for more information).
Note that it follows that $x \wedge \neg x = 0$ and $x \vee Dx = 1$.

In this paper we will only consider double p-algebras that are both \emph{regular} and \emph{bi-Stone}, that is, double-p algebras that satisfy both

$x \wedge Dx \leq x \vee \neg x$

\noindent and

$\neg x \vee \neg \neg x = 1$ and $Dx \wedge DDx = 0$.

\noindent Note that any of the last two equations imply both the equations $\neg Dx = \neg D \neg Dx$ and $D\neg x = D\neg D\neg x$. 

The notation ${\bf{3}}$ will stand for the three element bi-Stone and regular double p-algebra with universe $\{0 < \frac{1}{2} < 1 \}$.

This algebra has associated a propositional logic with 
connectives $\{ \land, \lor, \neg, D, \bot\}$ whose notion of semantic consequence is as follows.
A formula $\alpha$ is an \emph{algebraic consequence} of a set $\Gamma$ of formulas if for every valuation $v$ on ${\bf{3}}$, it holds that min$\{ v\gamma: \gamma \in \Gamma\} \leq v\alpha$.

Theorem 2 of \cite{MCTN} implies that the aforementioned propositional logic is sound and complete relative to a propositional calculus with the same rules for the connectives in {\bf{ID$_3$}}.

\subsection{Kripke semantics for {\bf{ID$_3$}}} \label{KS}

It holds that $\forall x(\alpha \lor Qx) \vdash_{ID_3} \alpha \lor \forall x Qx$, where $\alpha$ is a formula without ocurrences of free variables. 
For a proof, check part (ii) of the proof of Theorem 1 in \cite{MCTN}. 
As a consequence, it will be enough to consider Kripke models that have the same universe in every node, which are usually called ``Kripke models with constant domain".

\begin{defi}
Given a first-order language $L$, an L-\emph{Kripke structure} is a quadruple $(K, \leq, U, \rho)$ such that $(K, \leq)$ is a (non-empty) poset called \emph{frame}, $U$ is a non-empty set called \emph{universe}, $\rho$ is a binary function called \emph{realization} that assigns to each $n$-ary relation letter $R$ and $k \in K$ an $n$-ary relation $R^\rho_k \in U^n$ such that if $k \leq k'$, then $R^\rho_k \subseteq R^\rho_{k'}$.
\end{defi}

Given a L-Kripke structure with universe $U$, an \emph{assignment} is a function that assigns an element of $U$ to each variable in the language $L$. 
Given an assignment $e$, an $x$-\emph{variant assignment} of $e$ is an assignment $e^{u/x}$ such that $e^{u/x}(y) = u$ if $y = x$ else $e^{u/x}(y) = e(y)$, where $x, y$ are variables and $u \in U$.
We will use $E_U$ for the set of all the possible assignments in an L-Kripke structure with universe U.

\begin{defi}
An L-\emph{Kripke model} is a quintuple ${\bf{K}} = (K, \leq, U, \rho, e)$ such that $(K, \leq, U, \rho)$ is a $L$-Kripke structure and $e$ is an assignment.
\end{defi}

For any L-Kripke structure $(K, \leq, U, \rho)$, we write $F$ for the unique ternary relation $F \subseteq K \times E_U \times F_L$ satisfying the following conditions for $k, k' \in K$, $e \in E_U$, $R$ a relation letter in $L$, $x, x_1, \dots, x_n$ in the set of variables of $L$, and $\alpha$ and $\beta \in F_L$.

$(k, e, R(x_1, \dots, x_n)) \in F$ iff $(e(x_1), \dots, e(x_n)) \in R ^\rho_k$,

$(k, e, \alpha \land \beta)$ iff $(k, e, \alpha) \in F$ and $(k, e, \beta) \in F$,

$(k, e, \alpha \lor \beta)$ iff $(k, e, \alpha) \in F$ or $(k, e, \beta) \in F$,

$(k, e, \neg \alpha)$ iff for all $k' \geq k$, $(k', e, \alpha) \notin F$,

$(k, e, D\alpha)$ iff there exists $k' \leq k$ such that $(k', e, \alpha) \notin F$,

$(k, e, \forall x \alpha)$ iff for every node $k' \geq k$ and every $u \in U$ it holds that $(k',  e^{u/x}, \alpha) \in F$,

$(k, e, \exists x \alpha)$ iff there exists $u \in U$ such that $(k',  e^{u/x}, \alpha) \in F$.

\

For any L-\emph{Kripke model} with universe $K$ and assignment $e$, we define its associated forcing relation $\Vdash \subseteq K \times F_L$ by $(k, \alpha) \in \ \Vdash$ iff $(k, e , \alpha) \in F$. 
In what follows, we shall write $k \Vdash \alpha$ instead of $(k, \alpha) \in \ \Vdash$.

\begin{defi}
We say that a (closed) formula $\alpha$ is \emph{Kripke-consequence} of a set $\Gamma$ of (closed) formulas if for every Kripke model and every node $k$ it holds that if $k \Vdash \gamma$ for all $\gamma \in \Gamma$, then $k \Vdash \alpha$.

We say that a formula $\alpha$ is \emph{Kripke-valid} if for every Kripke model and every node $k$ it holds that $k \Vdash \alpha$.
\end{defi}

In the rest of this paper we will only consider Kripke models with universe $\{1 < \frac{1}{2}\}$.

\section{Rough sets semantics}

In this section we present another semantics for the logic {\bf{ID$_3$}}.

Rough sets were introduced by Pawlak and his co-workers in the early 1980s (for instance, see \cite{RS} and \cite{RSTARD}).

An \emph{approximation space} is a pair $(U, \theta)$, where $U$ is a non-empty set called the \emph{universe} of the approximation space and $\theta$ is an equivalence relation on $U$ called the \emph{indiscernibility} relation.

Given an approximation space $(U, \theta)$, we define the \emph{nth-power approximation space} of $(U, \theta)$ as the pair $(U^n, \theta^n)$, where $\theta^n$ is given by 

$((u_1, \dots ,u_n ), (v_1, \dots ,v_n )) \in \theta^n$ iff for all $1\leq i \leq n$, it holds that $(u_i, v_i) \in \theta$.

\noindent It is easily seen that $\theta^n$ is an equivalence relation (this construction already appears in \cite{TS}).

The following notions are central in the theory of rough sets.

\begin{defi}
Let ${\bf{A}} = (U, \theta)$ be an \emph{approximation space} and $X \subseteq U$. 

The \emph{lower approximation} of $X$ in ${\bf{A}}$, in symbols $\underline{X}$, is the set 

$\{u \in U :$ if there exists $x \in X$ such that $(u, x) \in \theta$, then $u \in X \}$. 

\noindent Analogously, the \emph{upper approximation} of $X$ in ${\bf{A}}$, in symbols $\overline{X}$, is the set 

$\{u \in U :$ there exists $x \in X$ such that $(u, x) \in \theta\}$.
\end{defi}

\

Let us now state our rough sets semantics.

\begin{defi}
Given a first-order language $L$ (which, for simplicity, we have assumed only with a non-empty set of $n$-ary predicate letters), a pair $(U, \sigma)$ where $U$ is a non-empty set and $\sigma$ is a function that associates an $n$-ary relation $\sigma(R) = R^{\sigma} \subseteq U^n$ to every $n$-ary predicate letter $R$ in $L$ will be called an \emph{$L$-structure}.
\end{defi}

Note that for a given approximation space $(U,\theta)$,  each $R^{\sigma}  \subseteq U^n$ may be viewed as a rough subset of $(U^n, \theta^n$). 

\begin{defi}
A \emph{rough $L$-structure} is a triple $(U,\theta, \sigma)$, where $(U, \theta)$ is an approximation space, $(U,\sigma)$ is an $L$-structure (and each $\sigma(R)$ is seen as a rough subset of $(U^n, \theta^n))$.
\end{defi}

\begin{defi}
A \emph{rough interpretation} of a language $L$ is a quadruple $\mathcal{I} = (U,\theta, \sigma, f)$, where $(U,\theta, \sigma)$ is a rough $L$-structure and $f : Var_{_L} \to U$ is a function assigning an element of $U$ to each variable of $L$. 
\end{defi}

As usual, given an interpretation $\mathcal{I} = (U,\theta, \sigma, f)$ and $a \in U$, 
the notation ${\mathcal{I}^{a/x}}$ indicates the interpretation with the same $L$-structure as $\mathcal{I}$ but with an assignment $f^{a/x}$ such that $f^{a/x}(x)= a$ and $f^{a/x}(y) = f(y)$, for $y \neq x$.

\

Recall that we indicate the upper approximation of $R^{\sigma}$ by $\overline{R^{\sigma}}$, and its lower approximation by $\underline{R^{\sigma}}$.

\begin{defi}
Let $\mathfrak{F}_{_L}$ be the set of formulas of the language $L$, let ${\bf{3}}$ be the three element algebra $(3; \wedge, \vee, \neg, D)$ as in the end of Subsection \ref{31} and let $\mathcal{I}$ be a rough interpretation for $L$ with assignment $f$.
We recursively define the function $v_{_\mathcal{I}}: \mathfrak{F}_{_L} \to \bf{3}$ which we will call the $\bf{3}$-\emph{valuation associated to} $\mathcal{I}$ as follows: 
\begin{itemize}
	\item[] For every $n$-ary predicate letter $R$, we stipulate
	$$v_{_\mathcal{I}}(R(x_1,\dots, x_n)):=
	\left\{
	\begin{array}{l}
		1, \textrm{ if } (f(x_1), \cdots, f(x_n)) \in \underline{R^{\sigma}}, \\
		\frac{1}{2}, \textrm{ if } (f(x_1), \cdots, f(x_n)) \in \overline{R^{\sigma}} - \underline{R^{\sigma}}, \\
		0, \textrm{ if } (f(x_1), \cdots, f(x_n)) \notin \overline{R^{\sigma}}.\\
	\end{array}
	\right.$$
    \item[] Let now $\alpha, \beta$ be $L$ formulas. 
We stipulate
    \begin{itemize}
    	\item[] $v_{_\mathcal{I}}(\neg \alpha) := \neg (v_{_\mathcal{I}}(\alpha))$,
    	\item[] $v_{_\mathcal{I}}(D \alpha) := D (v_{_\mathcal{I}}(\alpha))$,
    	\item[] $v_{_\mathcal{I}}(\alpha \wedge \beta) := v_{_\mathcal{I}}(\alpha) \wedge v_{_\mathcal{I}}(\beta)$, and 
    	\item[] $v_{_\mathcal{I}}(\alpha \vee \beta) := v_{_\mathcal{I}}(\alpha) \vee v_{_\mathcal{I}}(\beta)$.
    \end{itemize}
    \item[] Finally, for any $L$ formula $\alpha$ we define
    \begin{itemize}
    	\item[] $v_{_\mathcal{I}}(\forall x \alpha):= min\{v_{_{\mathcal{I}^{a/x}}}(\alpha): a \in U \}$ and 
    	\item[] $v_{_\mathcal{I}}(\exists x \alpha):= max\{v_{_{\mathcal{I}^{a/x}}}(\alpha): a \in U \}$. 
    \end{itemize}
\end{itemize}
\end{defi}

\begin{rem}
Pawlak at p.343 in \cite{RS} stated that ``\emph{we can interprete approximations as counterparts of necessity and possibility in modal logic}".
Let us note that the valuation associated to a rough interpretation $\mathcal{I}$ for the connectives $\Box$ and $\neg \neg$ (in the language of {\bf{I$\Box_3$}}) only takes values $0$ or $1$ and satisfies

$v_{_\mathcal{I}}(R(x_1,\dots, x_n)) = 1$ iff $(f(x_1),\dots, f(x_n)) \in \underline{R^{\sigma}}$ iff $(f(x_1),\dots, f(x_n))$ ``\emph{surely belongs}" to $R^\sigma$,
 
 $v_{_\mathcal{I}}(R(x_1,\dots, x_n)) = 1$ iff $(f(x_1),\dots, f(x_n)) \in \overline{R^{\sigma}}$ iff $(f(x_1),\dots, f(x_n))$ ``\emph{possibly belongs}" to $R^\sigma$.
 
\end{rem}

Now, let us define the notion of semantic consequence in the way studied in \cite{BEFGGTV}.

\begin{defi}
Let $\Gamma \cup \{\alpha\} \subseteq L$.
We define $\Gamma \vDash \alpha$ if for every interpretation $\mathcal{I}$ of $L$, it holds that min$\{v_I(\gamma) \} \leq v_I(\alpha)$.
\end{defi}

\section{Soundness and completeness}

Our goal is to prove soundness and completeness of the logic given in Section 2. 
In \cite{MCTN} we proved soundness and completeness relative to Kripke models as were given in Section \ref{KS}. 
So, it will be enough to prove that we can assign to every Kripke model a rough interpretation and conversely in such a way that Propositions \ref{ItoK} and \ref{KtoI} hold. 

\ 

To any rough interpretation we can associate a Kripke model as follows.

\begin{defi}
Let $\mathcal{I} = (U,\theta, \sigma, f)$ be a rough interpretation. 
We define the \emph{Kripke model associated} to the rough interpretation $\mathcal{I}$ as the Kripke model $K_{_\mathcal{I}} = (K, \leq, U_{_\mathcal{I}}, \rho, e)$ defined as follows. 
As the two-element Kripke models studied in \cite{MCTN}, $(K,\leq) = \{1 < \frac{1}{2}\}$. 
Its universe $U_{_\mathcal{I}}$ is the set of equivalence classes $\{[x]: x \in U \}$, 
the function $e(x) = [f(x)]$, and 
for every $n$-ary predicate letter we stipulate $\rho(R) = (R^\rho_1, R^\rho_{\frac{1}{2}})$, where

\begin{enumerate}
	\item[(A1)] $(e(x_1), \dots, e(x_n)) \in R^\rho_1$ iff $(f(x_1),\dots, f(x_n)) \in \underline{R^\sigma}$,
	\item[(A$\frac{1}{2}$)] $(e(x_1), \dots, e(x_n)) \in R^\rho_{\frac{1}{2}}$ iff $(f(x_1),\dots, f(x_n)) \in \overline{R^\sigma}$.	
\end{enumerate}
\end{defi}

It is possible to prove the following fact.

\begin{prop} \label{ItoK}

Let $\mathcal{I} = (U,\theta, \sigma, f)$ be a rough interpretation and $(K, \leq, U_{_\mathcal{I}}, \rho, e)$ its associated Kripke model.
For every formula $\alpha$ and every valuation $v$ it holds that

\

$v_{_\mathcal{I}}(\alpha) = 1$ iff $1 \Vdash \alpha$ \ and \ 
$\frac{1}{2} \leq v_{_\mathcal{I}}(\alpha)$ iff $\frac{1}{2} \Vdash \alpha$. 
\end{prop}

\begin{proof}
We check the cases of the atomic formulas, some connectives and the universal quantifier, leaving the rest for the reader.

$1 \Vdash R(x_1,\dots,x_n)$ iff 
$(e(x_1),\dots, e(x_n)) \in R^\rho_1$ if and only if  
$(f(x_1),\dots, f(x_n)) \in \underline{R^\sigma}$ if and only if $v_{_\mathcal{I}}(R(x_1,\dots,x_n)) = 1$.

$\frac{1}{2} \Vdash R(x_1,\dots, x_n)$ iff 
$(e(x_1),\dots, e(x_n)) \in R^\rho_\frac{1}{2}$ if and only if 
$(f(x_1),\dots, f(x_n)) \in \overline{R^\sigma}$ if and only if
$\frac{1}{2} \leq v_{_\mathcal{I}}(R(x_1,\dots,x_n))$.

Let us now suppose that the proposition holds for $\alpha$ and $\beta$. 
We have to prove that it holds for $\alpha \land \beta$. 
We have that $1 \leq v{_\mathcal{I}}(\alpha \land \beta)$ iff $1 \leq v{_\mathcal{I}}(\alpha) \land v{_\mathcal{I}}(\beta)$ iff $v{_\mathcal{I}}(\alpha) = 1$ and $v{_\mathcal{I}}(\beta) = 1$ iff $1 \Vdash \alpha$ and $1 \Vdash \beta$ iff $1 \Vdash \alpha \land \beta$. 

We also have that $\frac{1}{2} \leq v{_\mathcal{I}}(\alpha \land \beta)$ iff $\frac{1}{2} \leq v{_\mathcal{I}}(\alpha) \land v{_\mathcal{I}}(\beta)$ iff $\frac{1}{2} \leq v{_\mathcal{I}}(\alpha)$ and $\frac{1}{2} \leq v{_\mathcal{I}}(\beta)$ iff $\frac{1}{2} \Vdash \alpha$ and $\frac{1}{2} \Vdash \beta$ iff $\frac{1}{2} \Vdash \alpha \land \beta$.

Let us now suppose that the proposition holds for $\alpha$ and let us prove that it holds for $\neg \alpha$.

Since $v{_\mathcal{I}}(\neg \alpha) \neq \frac{1}{2}$, it is enough to note that $1 \leq v{_\mathcal{I}}(\neg \alpha)$ iff $v{_\mathcal{I}}(\alpha) = 0$ iff (by the inductive hypothesis) $1 \nVdash \alpha$ and $\frac{1}{2} \nVdash \alpha$ iff $1 \Vdash \neg \alpha$.

Since $v{_\mathcal{I}}(D \alpha) \neq \frac{1}{2}$, it is enough to note that $1 \leq v{_\mathcal{I}}(D \alpha)$ iff $v{_\mathcal{I}}(\alpha) = 0$ or $v{_\mathcal{I}}(\alpha) = \frac{1}{2}$ iff $1 \nVdash \alpha$ iff $1 \Vdash D\alpha$. 

Let us now suppose that the proposition holds for $\alpha$ and prove that it holds for $\forall x \alpha$.

Firstly, $v{_\mathcal{I}}(\forall x\alpha) = 1$ iff 
$min\{v_{_{\mathcal{I}^{a/x}}}(\alpha): a \in U \} =1$ iff 
for all $a \in U$, it holds that $v_{_{\mathcal{I}^{a/x}}}(\alpha) = 1$ iff 
for all $a \in U$, it holds that $v_{_{\mathcal{I}^{a/x}}}(\alpha) = 1$ and $\frac{1}{2} \leq v_{_{\mathcal{I}^{a/x}}}(\alpha)$ iff (by the inductive hypothesis) 
for all $a \in U$ it holds that $(1, e^{a/x}, \alpha) \in F$ and $(\frac{1}{2}, e^{a/x}, \alpha) \in F$ iff
$1 \Vdash \forall x \alpha$.

Secondly, $\frac{1}{2} \leq v{_\mathcal{I}}(\forall x\alpha)$ iff 
$\frac{1}{2} \leq min\{v_{_{\mathcal{I}^{a/x}}}(\alpha): a \in U \}$ iff 
for all $a \in U$, it holds that $\frac{1}{2} \leq v_{_{\mathcal{I}^{a/x}}}(\alpha)$ iff (by the inductive hypothesis)
for all $a \in U$ it holds that $(\frac{1}{2}, e^{a/x}, \alpha) \in F$ iff 
$\frac{1}{2} \Vdash \forall x \alpha$.
\end{proof}

\

Conversely, given a Kripke model of the form of those studied in \cite{MCTN}, we can associate a rough interpretation as follows.

\begin{defi}
Let ${\bf{K}} = (K, \leq, U, \rho, e)$ be a Kripke model with $(K, \leq) = \{1 < \frac{1}{2} \}$. 
We define the associated rough interpretation $\mathcal{I}_{\bf{K}}$ as follows.
\begin{enumerate}
	\item[] The universe of $\mathcal{I}_{\bf{K}}$ is the set $U' = U \times \{0,1\}$,
	\item[] relation $\theta$ is given by $(u,\varepsilon)\theta (v,\varepsilon')$ iff $u = v$ and 
	\item[] $f(x) = (e(x),0)$.	
	\item[] To any $n$-ary predicate letter $R$ in $L$ we associate the relation $R^{\sigma} \in (U')^n$ given by 
	\begin{eqnarray*}
		R^{\sigma}:= &
		\left\{
		((e(x_1), 0), \dots, (e(x_n), 0)) : (e(x_1), \dots, e(x_n)) \in R^\rho_\frac{1}{2}
		\right\} 
		\cup \ \ \ \ \ \ \ \ \ \ \ \ \ \ \ \ \ \ \ \ \ \ \ \ \ \ \ \ \ \ \ \ \ \ \ \ \ \ \ \ \ \ \  \\
		& \bigl\{
		((e(x_1),\varepsilon_1), \dots, (e(x_n),\varepsilon_n)) : (e(x_1), \dots, e(x_n)) \in R^\rho_1 \textrm{ and } \varepsilon_i \in \{0,1\} \textrm{ for } i \in \{1, \dots, n \}
		\bigr\}.
	\end{eqnarray*}
\end{enumerate}
\end{defi}

We can now prove the converse of Proposition \ref{ItoK}, that is, the following fact.

\begin{prop} \label{KtoI}

Let $(K, \leq, U, \rho, e)$ be a Kripke model and $\mathcal{I_K} = (U',\theta, \sigma, f)$ its associated rough interpretation.
Then, for every formula $\alpha$ and every valuation $v$ it holds that

\

$1 \Vdash \alpha$ iff $v_{_{\mathcal{I}_{\bf{K}}}}(\alpha) = 1$  and $\frac{1}{2} \Vdash \alpha$ iff $\frac{1}{2} \leq v_{_{\mathcal{I}_{\bf{K}}}}(\alpha)$.
\end{prop}

\begin{proof}
We check the cases of the atomic formulas and the universal quantifier, leaving the rest for the reader.

Let us check it for the case that $\alpha = R(x_1,\dots, x_n)$. 

Firstly, it holds that $1 \Vdash R(x_1,\dots, x_n)$ if and only if
$(e(x_1), \dots, e(x_n)) \in R^{\rho}_1$ if and only if  
$((e(x_1),\varepsilon_1), \dots, (e(x_n),\varepsilon_n)) \in R^\sigma$, for all $\varepsilon_i \in \{0,1\}$ iff 
$(f(x_1), \dots, f(x_n)) \in \underline{R^\sigma}$ if and only if
$v_{_{\mathcal{I}_{\bf{K}}}}(R(x_1,\dots,x_n)) = 1$. 

Secondly, $\frac{1}{2} \Vdash R(x_1,\dots, x_n)$ iff 
$(e(x_1), \dots, e(x_n)) \in R^{\rho}_\frac{1}{2}$ iff 
$(f(x_1), \dots, f(x_n)) \in R^{\sigma}$ if and only if 
$(f(x_1), \dots, f(x_n)) \in \overline{R^{\sigma}}$ iff 
$\frac{1}{2} \leq v_{_{\mathcal{I}_{\bf{K}}}}(R(x_1,\dots,x_n))$.

Let us now suppose that the proposition holds for $\alpha$ and deduce that it holds for $\forall x \alpha$.

Firstly, $1 \Vdash \forall x \alpha$ iff 
$(1,e,\forall x \alpha) \in F$ iff
for all $a \in U$ we have that $(1, e^{a/x}, \alpha) \in F$ iff (by the inductive hypothesis)
for all $a \in U$, it holds that $v_{\mathcal{I}_K^{(a,\epsilon)/x}} (\alpha) = 1$, for all $\epsilon \in \{0,1\}$ iff 
min $\{ v_{\mathcal{I}_{\bf{K}}^{a/x}}(\alpha) : a \in U, \epsilon \in \{ 0,1 \} \} = 1$ iff $v_{_{\mathcal{I}_{\bf{K}}}}(\forall x \alpha) = 1$.

On the other hand, $\frac{1}{2} \Vdash \forall x \alpha$ iff 
$(\frac{1}{2},e,\forall x \alpha) \in F$ iff 
for all $a \in U$ we have that $(\frac{1}{2},e^{a/x}, \alpha) \in F$ iff (by the inductive hypothesis)
for all $a \in U$, it holds that $v_{I_K^{(a,0)/x}} (\alpha) = 1$ iff
for all $a \in U$, for all $\epsilon \in \{0,1\}$, $\frac{1}{2} \leq v_{I_K^{(a,\epsilon)/x}} (\alpha)$ iff
$\frac{1}{2} \leq$ min $\{ v_{I_K^{(a,\epsilon)/x}}(\alpha) : a \in U, \epsilon \in \{0,1\} \}$ iff $v_{_{\mathcal{I}_{\bf{K}}}}(\forall x \alpha) = \frac{1}{2}$.
\end{proof}

Finally, we get the following result.

\begin{theo} \label{BT}
$\Gamma \Vdash \alpha$ if and only if $\Gamma \vDash \alpha$.
\end{theo}

\begin{proof}
Suppose there is a rough interpretation $\mathcal{I}$ such that $v_\mathcal{I} (\alpha) \leq  v_\mathcal{I} (\gamma)$ for all $\gamma \in \Gamma$. 
Then either $v_\mathcal{I}(\alpha) = 0$ or $v_\mathcal{I}(\alpha) = \frac{1}{2}$.
If $v_\mathcal{I}(\alpha) = 0$, then $\frac{1}{2} \leq v_\mathcal{I}(\gamma)$, for all $\gamma \in \Gamma$ whence $\frac{1}{2} \Vdash \gamma$ , for all $\gamma \in \Gamma$. It also holds that $v_\mathcal{I}(\alpha) = 0$ implies that $\frac{1}{2} \nVdash \alpha$. 
If $v_\mathcal{I}(\alpha) = \frac{1}{2}$, then $1 \Vdash \gamma$ , for all $\gamma \in \Gamma$ and $1 \nVdash \alpha$. 

Conversely, suppose there is a Kripke model such that either $1 \vDash \gamma$ for all $\gamma \in \Gamma$ and $1 \nvDash \alpha$ or $\frac{1}{2} \vDash \gamma$ for all $\gamma \in \Gamma$ and $\frac{1}{2} \nvDash \alpha$. 
In the first case, by Proposition \ref{KtoI} it follows that there is an interpretation $\mathcal{I}$ such that $v_\mathcal{I} (\gamma) = 1$ for all $\gamma \in \Gamma$ and $v(\alpha) \leq \frac{1}{2}$. 
In the second case, by Proposition \ref{KtoI} it follows that there is an interpretation $\mathcal{I}$ such that $\frac{1}{2} \leq v_\mathcal{I} (\gamma)$ for all $\gamma \in \Gamma$ and $v(\alpha) = 0$. 
\end{proof}

\begin{cor}
The logic $ID_3$ is sound and complete relative to the rough sets semantics.
\end{cor}

\begin{proof}
By Theorem \ref{BT} and the fact that in \cite{MCTN} we proved that $ID_3$ is sound and complete relative to the two-element Kripke models considered above.
\end{proof}


\end{document}